\documentclass[11pt]{amsart}

\usepackage{amscd}
\usepackage{amsmath}
\usepackage{amssymb}
\usepackage{amsthm}
\usepackage{epsf}
\usepackage{latexsym}
\usepackage{verbatim}
\input epsf.tex

\newcounter{thmcount}
\newtheorem{theorem}[thmcount]{Theorem}
\newtheorem{definition}{Definition}[section]
\newtheorem{proposition}{Proposition}[section]
\newtheorem{lemma}[proposition]{Lemma}
\newtheorem{corollary}[proposition]{Corollary}
\newtheorem*{remark}{Remark}

\theoremstyle{definition}

\newcommand{\R}{\mathbb{R}}

\newcommand{\Q}{\mathbb{Q}}

\newcommand{\N}{\mathbb{N}}

\newcommand{\ve}{\varepsilon}

\begin{document} 

\title{On $\mu$-Compatible Metrics and Measurable Sensitivity}

%\title{A Complete Classification of Finite Measure Preserving Ergodic Transformations Based on W-Measurable Sensitivity}

\author[Grigoriev]{Ilya Grigoriev}
\address[Ilya Grigoriev]{Department of Mathematics\\Stanford University\\Stanford, CA 94305 USA}
%{  University of Chicago\\ Chicago, IL 60637, USA }
\email{ilyagr@stanford.edu}

\author[Iordan]{Marius C\u{a}t\u{a}lin Iordan}
\address[Marius C\u{a}t\u{a}lin Iordan]{Williams College\\ Williamstown, MA 01267, USA}
\email{mci@cs.stanford.edu}

\author[Lubin]{Amos Lubin}
\address[Amos Lubin]{Harvard College\\ University Hall \\
Cambridge, MA 02138, USA}
\email{lubin@math.berkeley.edu}

\author[Ince]{Nathaniel Ince}
\address[Nathaniel Ince]{Massachusetts Institute of Technology\\ 77 
Massachusetts Ave.\\
Cambridge, MA 02139-4307, USA}
\email{incenate@gmail.com}

\author[Silva]{Cesar E. Silva}
\address[Cesar Silva]{Department of Mathematics\\
     Williams College \\ Williamstown, MA 01267, USA}
\email{csilva@williams.edu}

%\date{\today}

\begin{abstract}
  We introduce the notion of W-measurable sensitivity, which extends
  and strictly implies canonical measurable sensitivity, a
  measure-theoretic version of sensitive dependence on initial
  conditions. This notion also implies pairwise sensitivity with
  respect to a large class of metrics. We show that nonsingular 
   ergodic and conservative  dynamical systems on standard spaces must be either
  W-measurably sensitive, or isomorphic mod 0 to
  a minimal uniformly rigid isometry. In the finite measure-preserving case 
  they are  W-measurably sensitive or measurably  isomorphic to an ergodic isometry on
  a compact metric space.
\end{abstract}

\maketitle

% Section 1 %%%%%%%%%%%%%%%%%%%%%%%%%%%%%%%%%%%%%%%%

\section{Introduction}

The notion of sensitive dependence on initial conditions is an extensively studied isomorphism invariant  of topological dynamical systems on compact metric spaces (\cite{GW93}, \cite{A-B96}). In \cite{J-S08}, the authors define two measure-theoretic versions of sensitive dependence, measurable sensitivity and strong measurable sensitivity, and show that, unlike their traditional topologically-dependent  counterpart, both of these properties carry up to measurable-theoretic isomorphism. James et. al. introduce these notions for nonsingular transformations and show that measurable sensitivity is implied by double ergodicity (a property equivalent to weak mixing in the finite measure-preserving case) and strong measurable sensitivity is implied by light mixing in the finite measure-preserving case. 

%They also classified strong measurable sensitive transformations, showing in particular that an ergodic infinite measure-preserving transformation cannot be strong measurably sensitive.

In this paper, we introduce W-measurable sensitivity, a notion that is
{\it a priori} stronger than measurable sensitivity and implies it
straightforwardly. We use this new  property, together with
properties of $\mu$-compatible metrics (see below), to formulate a  
classification of all nonsingular conservative and  ergodic
transformations on standard Borel spaces as being either W-measurably
sensitive or isomorphic to a minimal uniformly rigid isometry; in the case
of finite invariant measure we obtain more, namely W-measurably
sensitive or isomorphic to a minimal uniformly rigid invertible isometry on a compact metric
space. In the course of this proof, we
also show that W-measurable sensitivity is in fact equivalent to
measurable sensitivity for conservative and ergodic transformations.

In addition, we show (see Appendix A) that the notion of W-measurable
sensitivity is closely related to {\it pairwise sensitivity}, a notion
introduced in \cite{cadre2005ps} for finite measure-preserving transformations. In their paper, Cadre and Jacob show
that weakly mixing finite measure-preserving  transformations always exhibit pairwise
sensitivity, and also  any ergodic  finite measure-preserving transformation satisfying a
certain entropy condition. Our results imply that any   finite measure-preserving ergodic
transformation that is not isomorphic mod 0 to a Kronecker
transformation will exhibit pairwise sensitivity with respect to any
$\mu$-compatible metric (in addition to W-measurable sensitivity).

The plan of the paper is as follows.  Section~\ref{S:defs} recalls basic definitions from \cite{J-S08} and introduces $\mu$-compatible metrics and some of their properties. In Section~\ref{sec:w-meas-sens} we define
W-measurable sensitivity.  Section~\ref{S:Lipshitz Metrics} starts by construncting 
$1$-Lipshitz metrics from any  metric on a dynamical system, and then shows that W-measurable sensitivity can be equivalently expressed in
  additional ways using properties of $\mu$-compatible metrics. In 
Section~\ref{sec:cond-1-lip-mu-comp}, we provide a sufficient
condition under which the newly constructed $1$-Lipshitz metric is in
fact $\mu$-compatible, and discuss consequences of this fact largely from \cite{A-G01}.  In Section~\ref{sec:Constructing-mu-comp.-metr-on-isom.-copies} we discuss the invariance of W-measurable sensitivity under measurable isomorphism, as well as the technical assumptions necessary
for it to hold. We also illustrate the main connection between
$1$-Lipshitz metrics and W-measurable sensitivity, namely that a
conservative and ergodic nonsingular dynamical system is W-measurably
sensitive if and only if all dynamical systems $(X',\mu',T')$
isomorphic mod 0 to it admit no $\mu'$-compatible $1$-Lipshitz
metrics. Finally, in Section~\ref{sec:meas-sens-fin-measure} we prove our main result, which
classifies all conservative and ergodic, nonsingular  transformations on 
standard Borel spaces  as being either W-measurably sensitive, or isomorphic
to a minimal uniformly rigid invertible isometry. A corollary of this fact is that for conservative and ergodic  transformations, W-measurable sensitivity is
equivalent to  measurable sensitivity as defined in
\cite{J-S08}.  We end the section by obtaining a stronger result in the case of ergodic
finite measure-preserving transformations. 

In  Appendix~\ref{sec:pairwise-sensitivity} elaborates on the relationship between our results and the notion of pairwise sensitivity as introduced in \cite{cadre2005ps} and mention
the recent work in \cite{H-L-Y2010}.

\subsection{Acknowledgements}
This paper is based on research by the Ergodic Theory group of the
2007 SMALL summer research project at Williams College.  Support for
the project was provided by National Science Foundation REU Grant DMS
- 0353634 and the Bronfman Science Center of Williams College. The first-named
author would also like to acknowledge support by an NSF graduate fellowship.

We are indebted to the referee   for a careful  reading of the manuscript and several comments and suggestions
that improved our paper. We thank Ethan Akin for several remarks  including 
 an argument that removed the assumption of forward measurability in an 
earlier version of our paper, the proof of Proposition \ref{P:invertibleisometry},   and for bringing \cite{H-L-Y2010} to our attention.
% Section 2 %%%%%%%%%%%%%%%%%%%%%%%%%%%%%%%%%%%%%%%%

\section{Preliminary Definitions}\label{S:defs}

A {\it nonsingular dynamical system} is a quadruple $(X,\mathcal {S}(X),\mu,T)$,
where $(X,\mathcal {S}(X),\mu)$ is 
 a standard nonatomic Lebesgue space (i.e., $(X,\mathcal {S}(X))$ is a standard
 Borel space, see e.g. \cite{Sri98}, and  $\mu$ is a 
$\sigma$-finite, nonatomic measure on $\mathcal{S}(X)$).   It follows that $X$ must be of cardinality $c$ as the 
measure is nonatomic.   Furthermore, the transformation 
$T$ is  measurable and  a nonsingular endomorphism (i.e.,  for all
$A \in \mathcal{S}(X)$, $T^{-1}(A)\in\mathcal S(X)$ and
 $\mu(A)=0$ if and only if $\mu(T^{-1}(A))=0$, see e.g. \cite{Si08}). In some cases we  assume that $T$ is
measure-preserving or that the measure space is   finite. 
Recall that $T$ is conservative and ergodic if and only if for all measurable sets $A$,
if $T^{-1}(A)\subset A$, then $\mu(A)=0$ or $\mu(A^c)=0$.

We consider metrics or pseudo-metrics on $X$.  We assume throughout this article that all pseudo-metrics
$d:X\times X\to\mathbb R$ are (Borel) measurable and bounded by $1$ (one can replace
$d$ by $\frac{d}{1+d}$). It follows that, for each $\ve>0$, the set  $\{(x,y)\in X\times X: d(x,y)<\ve\}$ 
is measurable. Therefore, by e.g. \cite[Exercise 3.1.20]{Sri98}, the balls
\[
B^d(x,\ve)=\{y\in X: d(x,y)<\ve\}
\] are measurable.
For a pseudo-metric $d$ define
\begin{align*}
\mathcal D^d(x)&=\max\{\ve\geq 0: \mu(B^d(x,\ve))=0\} \text { and }\\
Dis(d)&=\{x\in X: \mathcal D^d(x)>0\}.
\end{align*}
%When the pseudo-metric $d$ is clear from the context we may write $\mathcal D$ for $\mathcal D^d$.

A (measurable) metric $d$ on $(X,\mathcal{S}(X),\mu)$ is said to be {\it
  $\mu$-compatible} if $\mu$ assigns positive (nonzero) measure to all
nonempty, open $d$-balls in $X$, equivalently if $Dis(d)=\emptyset$, or if $\mathcal D^d(x)=0$ for 
all $x\in X$.  If $d$ a $\mu$-compatible metric on
$(X,\mathcal{S}(X),\mu)$, then $X$ is separable under $d$
(see \cite[1.1]{J-S08} and Proposition~\ref{P:metricprop} below). Therefore open sets are measurable 
as they are countable unions of balls.  All $d$-closed sets are also
measurable, etc. We say that $d$ is {\it $\mu$-separable} if  $\mu(Dis(d))=0$,
or equivalently $\mathcal D^d(x)=0$ a.e.
If follows that if $d$ is $\mu$-separable, then the restriction of $d$ to
$X\setminus Dis(d)$ is $\mu$-compatible.

%A metric $d$ is {\it discrete} if there exisits $\delta>0$ such that
% $\mathcal D^d(x)> \delta$ for all $x\in X$ (similarly one defines
%{\it  $\delta$-discrete } for a fixed $\delta$).

\begin{proposition}\label{P:metricprop} Let $(X,\mathcal{S}(X),\mu,T)$ be a nonsingular dynamical system
and let $d$ be a pseudo-metric on $X$. 
\begin{enumerate}
\item  The function $\mathcal D^d(x)$ is continuous with respect to $d$ and measurable.
\item  \label {e:separable} The pseudo-metric $d$ is separable when restricted to  
  $X\setminus Dis(d)$. In particular, if $d$ is $\mu$-compatible,
then it is separable on $X$.
\item  $Dis(d)$ is open with respect to $d$ and measurable.
\item A pseudo-metric $d$ is $\mu$-separable if and only if there exists a measure zero subset $Z$ of $X$ 
such that $d$ restricted to $X\setminus Z$ is separable.
%In particular, if $d$ is a separable pseudo-metric on $X$, then  $\mu(Dis(d))=0$.
\end{enumerate}
\end{proposition} 

\begin{proof}
(1) Suppose  that  $\beta< \mathcal D^d(x) < \alpha$.
Set
\[\delta=\frac12 \min\{\mathcal D^d(x)-\beta,\alpha-\mathcal D^d(x)\}.
\] The for each $y\in B^d(x,\delta)$ we have
\begin{align*}
B^d(y,\beta)&\subset B^d(x,\beta+\delta),\\
B^d(x,\alpha-\delta)&\subset B^d(y,\alpha).
\end{align*}
Since $\beta+\delta<\mathcal D^d(x)$, $\mu(B^d(x,\beta+\delta))=0$, so $\mu(B^d(y,\beta))=0$
and 
$\mathcal D^d(y) \geq \beta$. Similarly we obtain that  $\mathcal D^d(y)\leq \alpha$.
This implies that  $\mathcal D^d$ is continuous with respect to $d$, and therefore measurable.

(2) For $0<\ve<1$, let $A_\ve\subset X\setminus Dis(d)$ be such that
if $x, y\in A_\ve$ then $d(x,y)\geq \ve$ and let it be maximal with respect to this property.
It follows that \[\{B^d(x,\ve/2): x\in A_\ve\}\] is a collection of disjoint sets
of positive measure and since $\mu$ is $\sigma$-finite, this collection is countable. 
This shows that each $A_\ve$ is countable. Then the union $\bigcup A_{1/n}$,  for $n\in\N$,
is  a countable set that is dense in $X\setminus \mathcal D^d$ for the metric $d$.
 
 (3) Since $\mathcal D^d$ is continuous by part (1),  $Dis(d)$ is open with respect to $d$.
 By part \eqref{e:separable}, every open set that is contained in $X\setminus Dis(d)$ 
 is a countable union of balls, hence it is measurable. Similarly, closed sets contained in 
 $X\setminus Dis(d)$ are measurable. In particular, $X\setminus Dis(d)$, and so 
 $Dis(d)$, are measurable.

(4) Suppose that   $\mu(Dis(d))>0$ and let $Z\subset X$ be such that
$\mu(Z)=0$.  We show that $d$ is not separable on the  subset $Dis(d)\setminus Z$
of $X\setminus Z$. We first note that 
 the collection 
\[\{B^d(x,\mathcal D(x)):  x\in Dis(d)\setminus Z\}\]
is an open cover of $Dis(d)\setminus Z$, and since $Dis(d)\setminus Z$ has positive measure
and each of the balls has measure zero (by definition of $Dis(d)$), the collection 
cannot have 
a countable subcover. Conversely, if $\mu(Dis(d))= 0$ we can let $Z=Dis(d)$ and 
use part \eqref{e:separable}.
\end{proof}

\begin{proposition}\label{P:metricpropD} Let $(X,\mathcal{S}(X),\mu,T)$ be a nonsingular dynamical system
and let $d$ be a pseudo-metric on $X$.  Let $\delta>0$.
If $\mathcal D^d(x)\geq \delta $ for almost all $x\in X$, then
$\mathcal D^d(x)\geq \frac{\delta}{2}$ for all $x\in X$.
\end{proposition}

\begin{proof}  Let 
\[Z=\{z\in X: \mathcal D^d(z)\geq \delta\} = \{z\in X: \mu(B^d(z,\delta))=0\}.
\]We know that $\mu(Z^c)=0$.
Suppose  $\mathcal D^d(x)<  \frac{\delta}{2}$ for some $x\in X$. Then
$\mu(B^d(x,\delta/2))>0$. So there exists $z\in B^d(x,\delta/2) \cap Z$.
By the triangle inequality,  $B^d(x,\delta/2)\subset B^d(z,\delta)$.
This means that $\mu(B^d(z,\delta))>0$, a contradiction.
\end{proof}

\section{W-measurable Sensitivity}\label{sec:w-meas-sens}

We start by recalling the definition of measurable sensitivity. 

\begin{definition} \cite{J-S08} \label{t:wmsdef}
 A  nonsingular dynamical system  $(X,\mathcal{S}(X),\mu,T)$ is said to be 
 \emph{measurably sensitive} if for every  isomorphic mod $0$ dynamical
system $(X_1,\mathcal{S}(X_1),\mu_1,T_1)$   and any  $\mu_1$-compatible metric 
 $d$  on $X_1$, then
there exists $\delta
> 0$ such that for   $x \in X_1$ and all  $\varepsilon > 0$ there exists
$n \in \mathbb{N}$ such that \[\mu_1 \{y\in B_{\varepsilon}(x):
d(T_1^n(x),T_1^n(y)) > \delta \} > 0.\]
\end{definition}

We now introduce the definition that we shall be using extensively.

\begin{definition}
  \label{def:Lightly-meas-sens} For a $\mu$-compatible metric $d$, a
  nonsingular dynamical system $(X,\mu,T)$ is \emph{W-measurably
    sensitive} \emph{with respect to} $d$ if there is a $\delta > 0$
  such that for every $x \in X$,
  \[ \limsup_{n\to\infty}d(T^{n}x,T^{n}y)>\delta\]
  for almost every $y\in X$. The dynamical system is said to be
  \emph{W-measurably sensitive} if the above definition holds true for
  all $\mu$-compatible metrics $d$.
\end{definition}

\noindent{\bf Remark.}
(1) As in \cite{J-S08},
  % and
  %\cite{cadre2005ps}\sidenote{Are the proofs actually similar??}, 
  it
  can be shown that a doubly ergodic nonsingular transformation is W-measurably
  sensitive.  (Double ergodicity is a condition for nonsingular transformations that
  is  equivalent to weak
  mixing in the finite measure-preserving case \cite{Fu81}.)
  There exist both infinite (and finite) measure-preserving 
and nonsingular type III (i.e., not admitting an equivalent $\sigma$-finite invariant measure) invertible transformations that are doubly ergodic
(see e.g. \cite{DS09}), and therefore W-measurably sensitive.
 
(2)  \label{sec:remark-nonatomic}
 If a measure space
  $(X,\mu)$ has atoms, no transformation on it
  can exhibit W-measurable sensitivity with respect to any metric.
  Indeed, for any $x \in X$, and any $\delta$, the set of points $y$
  such that $\limsup_{n \to \infty} d(T^n x, T^n y) > \delta$ cannot
  include $x$. So  this set cannot have full measure (i.e., its complement has measure zero)  if
  $\mu(\{x\}) >0$.

  The same is not true about  measurable sensitivity. For this
  reason, throughout this paper we  assume that our measure
  space is nonatomic. 
  %This assumption is justified because most of
 % the results of this paper concern ergodic transformations, and
 % ergodic transformations on spaces with points of positive measure
  %are all rotations on a discrete set.
  %\sidenote{This is right, isn't it??}

 (3)  \label{sec:invertible-examples} A very important example of an ergodic finite
 measure-preserving  dynamical system which is not
  W-measurably sensitive is a Kronecker transformation, i.e. an
  ergodic isometry on an interval of finite length (with the Lebesgue
  measure and the usual metric).  This transformation is not
  W-measurably sensitive with respect to the usual metric because it
  is an isometry.  There are also examples of  conservative and ergodic type III nonsingular invertible transformations
that are not W-measurable sensitive. Let $X=\Pi_{i=0}^\infty \{0,1\}$, the $2$-adic integers,  let $T$ addition by 1,
$Tx=x+1$, and $d$ be the $2$-adic metric. Then it is well known that $T$ is  a minimal isometry for $d$. Let $0<p<1$
and $\mu_p=\Pi_{i=0}^\infty\{p,1-p\}$, a probability measure on the Borel
$\sigma$-field $\mathcal B$. Then $\mu_p$ is a nonsingular measure for $(X,\mathcal B,T)$ that is conservative
and ergodic of type III (when $p\neq 1/2$), see e.g. \cite{DS09}. It is clear
that $d$ is $\mu_p$-compatible, so  $(X,\mathcal B,T)$ is a conservative 
ergodic invertible nonsingular transformation that is not finite measure-preserving
and is not W-measurably sensitive.

\bigskip
 
We note that, the property of W-measurable sensitivity is
preserved under measurable isomorphisms (Proposition~\ref{prop:light-meas-sens-and-meas-isomorph}). 

%Let us make a rough comparison of the two definitions
%above. Definition~\ref{def:Lightly-meas-sens} states that for some
%$\delta > 0$, every $x \in X$ satisfies the condition that for a set
%of \emph{full measure} $Y \subset X$, for every $y \in Y$, there are
%\emph{infinitely many} values of $n$ for which
%$d(T^{n}x,T^{n}y)>\delta$. On the other hand, measurable sensitivity as
%defined in~\cite{J-S08} only requires that there is a $\delta>0$ for
%which every $x \in X$ satisfies the condition that, for every
%$\varepsilon>0$, there is a set of \emph{positive measure}
%$Y \subset B_\varepsilon(x)$ such that for every $y\in Y$ there is
%\emph{one} value of $n$ for which there is a set of positive measure
%$Y \subset B_\varepsilon(x)$ and a single $n$, depending only on $Y$,
%such that $d(T^n x, T^n y) > \delta$. 
 W-measurable sensitivity clearly implies measurable sensitivity (see first part of  the proof of
Proposition~\ref{pro:equiv-weak-and-light}).
%  (this description differs
% slightly from Definition~\ref{t:wmsdef} as in that definition, the
% value of $n$ depends only on the choice of $Y$, and not on the choice
% of $y\in Y$; however, the two notions are equivalent, as shown in the
% proof of Proposition~\ref{pro:equiv-weak-and-light}).
In fact, we show that the two notions are equivalent for conservative and ergodic dynamical
systems. We first show in Proposition~\ref{pro:mild-meas-sens2} that
for a transformation to be W-measurably sensitive, it is sufficient
for each $y \in Y$ to have \emph{one} value of $n$ that satisfies
$d(T^{n}x,T^{n}y)>\delta$. The remainder of the equivalence follows
from the results in the following sections, culminating with
Proposition~\ref{pro:equiv-weak-and-light}.

\section{Constructing $1$-Lipshitz Metrics}
\label{S:Lipshitz Metrics}

We shall use the term \emph{$1$-Lipshitz metrics (with respect to $T$)} to denote metrics that satisfy the inequality $d(Tx,Ty)\leq d(x,y)$
for all $x$ and $y$.

First, we provide a way to construct a $1$-Lipshitz metric from any other metric.

\begin{definition}
Let $(X,\mu,T)$ be a nonsingular dynamical system, and $d$ be a  
metric on $X$. Define, for   $x,y\in X$,  
 \[
d_{T}(x,y)=\sup_{n\geq0}d(T^{n}x,T^{n}y).\]

%We shall use the notation $B_{T_{\delta}}(x)$ for the set $\{ y\mid d_{T}(x,y)<\delta\}$.
\end{definition}

\begin{lemma}
$d_{T}$ is a metric on $X$ (satisfying our standing assumptions: measurable 
and bounded). Moreover, it is a $1$-Lipshitz metric.
\end{lemma}

\begin{proof}
The first statement is left to the reader. To see that it is $1$-Lipshitz we compute,
\begin{align*}
d_T(Tx,Ty)&=\sup_{n\geq0}d(T^{n}(Tx),T^{n}(Ty)) = \sup_{n\geq1}d(T^{n}x,T^{n}y) \\
&\leq \sup_{n\geq0}d(T^{n}x,T^{n}y) =d_T(x,y).
\end{align*}
\end{proof}

\noindent{\bf Remark.}
In general, even if the metric $d$ is $\mu$-compatible,
the metric $d_{T}$ may not be  $\mu$-compatible. Consequently, there is no guarantee that
the measure space is separable under the topology determined by $d_T$.
%, and it is conceivable that there are $d_T$-open sets  which are not measurable. - Not really, see above lemma

For example, let $I$ be the unit interval, $\lambda$ be the Lebesgue
measure, and $d$ be the usual metric. Let $T:I\to I$ be the doubling
map $Tx=2x\pmod1$. Note that $d$ is a $\lambda$-compatible metric.

The metric $d_{T}$, however, is not $\lambda$-compatible. Indeed,
for any $x\not\in\mathbb{Q}$, and any $\varepsilon>0$, there will
be an $n$ such that $d(0,T^{n}x)>1-\varepsilon$. So, since $T(0)=0$,
we have\[
\sup_{n\geq0}d(T^{n}(0),T^{n}y)=\sup_{n\geq0}d(0,T^{n}y)=1.\]

In other words, for any $0<\delta<1$, the $\delta$ ball around 0
in the $d_{T}$ metric may contain only rational points. So,
$\lambda(B_{T_{\delta}}(0))=0$, and $d_T$ is not $\lambda$-compatible.

In this example, the transformation $T$  turns out to be
W-measurably sensitive. In fact, since  mixing, it is strongly
measurably sensitive (see \cite{J-S08}). On the other hand, we will
see that whenever the $1$-Lipshitz metric $d_T$ is $\mu$-compatible,
the corresponding transformation $T$ is not W-measurably sensitive.

%We will also see a sort of a converse to this statement in Proposition~\ref{thm:non-meas-sens}.
\bigskip %End of Remark

We now formulate several equivalent definitions of W-measurably sensitive 
transformations. We start by showing that while the original definition 
requires the existence of infinitely many times $n$ satisfying the condition, it
is sufficient to require the existence of one such $n$.

\begin{proposition}
\label{pro:mild-meas-sens2} Let $(X,\mu,T)$ be a nonsingular dynamical
system, and $d$ be a $\mu$-compatible metric. The following are
equivalent:
\begin{enumerate}
 \item The system is W-measurably sensitive with respect to $d$.
\item There is a $\delta>0$ such that, for each $x\in X$, for almost
every $y\in X$, 
\[d_T(x,y)> \delta.\]
\item There is a $\delta>0$ such that for each  $x\in X$,
\[
\mu(B^{d_T}(x,\delta))=0.
\]
\item There is a $\delta>0$ such that
for each  $x\in X$, 
\[\mathcal D^{d_T}(x)\geq \delta.\]
\item There is a $\delta>0$ such that for each  $x\in X$, 
\[\mathcal D^{d_T}(x) > \delta.\]
\end{enumerate}
\end{proposition}

\begin{proof} 
  $(2) \Rightarrow (1)$.
  Suppose that  there is a $\delta>0$ such that for each  $x\in X$, 
   for
  almost every $y\in X$, there exists $n$ such that
  $d(T^{n}x,T^{n}y) > \delta$. For every
  natural number $N$ and $x\in X$ define a set $Y(N,x)$ by:
  \[  
  Y(N,x) = \{y \in X:   \exists n > N,\: d(T^n x,T^n y) > \delta \}.
   \]

  We now prove that for all $N$ and $x$, the set $Y(N,x)$ has full
  measure. Consider the point $T^N x$. Using our assumption,
  for almost every $y\in X$, there exists $n$ such that
  $d\left( T^{n}(T^N x),T^{n}y \right) > \delta$. In other words, the set
  \[  
  Z(N,x) = \{y \in X:    \exists n > 0,\: d(T^{N+n} x,T^n y) > \delta\} \]
  has full measure.
Notice that $Y(N,x) = T^{-N}(Z(N,x))$. Since $T$ is a nonsingular
  transformation, $Y(N,x)$ must also have full measure.

  Finally, let $Y_x = \bigcap_{N=0}^{\infty} Y(N,x)$. Clearly,  $Y_x$
  has full measure. Furthermore, for every $y \in Y_x$, there
  are infinitely many values of $n$ such that 
  $d(T^{n}x,T^{n}y) > \delta$. So  
  \[
  \limsup_{n\to\infty} d(T^{n}x,T^{n}y) \geq \delta \]
  for almost all $y \in X$.
  Therefore the system $(X,\mu,T)$
  is W-measurably sensitive with respect to $d$.

  $(1) \Rightarrow (2)$.
 The converse is clear from the definitions.

  $(2) \Leftrightarrow (3)$. 
  % Is this clear enough? I really want to make it clear to the reader
  % that 2) and 3) are really different ways of saying exactly the
  % same thing. 
  If condition $(2)$ is satisfied at $x$ for some $\delta$, then $B^{d_T}(x,\delta)$ is
  contained in the complement of a set of full measure. So  $\mu(B^{d_T}(x,\delta))=0$.

  Conversely, if condition $(3)$ is satisfied at $x$ for some $\delta$, then
  $B^{d_T}(x,\delta)$ has measure zero. So  in particular, the set
  $\{ y \in X: 
  \forall n \geq 0, \: d(T^{n}x,T^{n}y) \leq  \delta/2 \}$ 
  has measure zero. Therefore, for almost every $y \in X$, there is
  some $n$ for which $d(T^nx,T^ny)>\delta / 2$, and condition $(2)$ is
  satisfied.
  
   The equivalence of $(3)$ and $(4)$ is clear form the definitions.
  The equivalence of $(4)$ and $(5)$ is clear since $\delta$ does not have to be the same.
\end{proof}

\noindent {\bf Remark.} From Proposition~\ref{P:metricpropD} it follows that 
in the equivalent characterizations of W-measurable sensitivity in Proposition~\ref{pro:mild-meas-sens2},
one can replace ``for each $x\in X$'' in parts $(2) - (5)$ with ``for a.e. $x\in X$."

%%%%%%%%%
\section{Conditions for $1$-Lipshitz metric $d_T$ to be
  $\mu$-compatible and Consequences}\label{sec:cond-1-lip-mu-comp}

Now, we provide a sufficient condition for the $1$-Lipshitz metric
$d_{T}$ to be $\mu$-compatible given that the transformation $T$ is
ergodic. 

%%%%%
%%%%%
\begin{comment}
First, we shall prove a technical lemma:

\begin{lemma}
  \label{lem:stupid-containment} Suppose $d$ is a $1$-Lipschitz metric
  on $X$ with respect to a transformation $T$.  Then, for any integer
  $m>0$ and any real number
  $r>0$,\[ B_{r}(x)\subset T^{-m}\left(B_{r}(T^{m}x)\right).\]
\end{lemma}

\begin{proof}
  Suppose $y\in B_{r}(x)$. Then $d(x,y)<r$. Since $d$ is $1$-Lipshitz,
  we must have $d(T^{m}x,T^{m}y)<r$ for all positive integers $m$. So 
  $T^{m}y\in B_{r}(T^{m}x)$ and therefore
  $y\in T^{-m}\left(B_{r}(T^{m}x)\right)$.
\end{proof}
\end{comment}
%%%%%
%%%%%

The proof of the following lemma is standard, see for example 
\cite[Corollary 2.7]{ST91}.

\begin{lemma}\label{L:subinvatiantfunc}
Let $(X,\mu,T)$ be a conservative and
  ergodic nonsingular dynamical system.  Let $f:X\to\R$ be
  a measurable function. If $f\geq f\circ T$ a.e., then 
  $f=f\circ T$ a.e.
\end{lemma}

\begin{lemma}\label{L:Dineq}
Let $(X,\mu,T)$ be a nonsingular dynamical system, and $d$ be a  
metric on $X$. If $d$ is $1$--Lipshitz then
\[\mathcal D^d \geq \mathcal D^d\circ T \text{ on } X.
\]
\end{lemma}

\begin{proof} Let $T^*d$ denote the metric $T^*d(x,y)=d(Tx,Ty)$.
First we observe
\begin{align*}
T^{-1} B^d(Tx,\ve)&=\{y\in X: d(Tx,Ty)<\ve\}=B^{T^*d}(x,\ve).
\end{align*}
Since $T$ is nonsingular, $\mu(B^{T^*d}(x,\ve))=0$ if and only if
$\mu(B^{d}(Tx,\ve))=0$. It follows that
\[
\mathcal D^{T^*d}(x)=\mathcal D^d(Tx)\text{ for all }x\in X.
\]
Since $T$ is  $1$-Lipshitz, $d(x,y)\geq d(Tx,Ty)$, which implies
\[
\mathcal D^d(x) \geq \mathcal D^{T^*d}(x)\text { for all }x,
\] completing the proof.
\end{proof}

Now, we are ready to state the sufficient condition the $1$-Lipshitz
metric $d_{T}$ to be $\mu$-compatible which is our main tool in proving the main results in Section~\ref{sec:meas-sens-fin-measure}.
% we will use in the nextsection.

\begin{lemma}
  \label{lem:when-d'-is-mu-comp}   Let $(X,\mu,T)$ be a conservative and
  ergodic nonsingular dynamical system. Let $d$ be a $\mu$-compatible metric on $X$. Suppose
  further that 
$T$ is not W-measurably sensitive with respect to $d$.
   Then 
there exists a positively invariant measurable set $X_1$ of full measure
(i.e., $X_1\subset T^{-1}(X_1)$  and $\mu(X\setminus X_1)=0$)   
   such that $d_T$ is a $\mu$-compatible
  metric for the system  $(X_1,\mu,T)$, where  $\mu$ 
  and $T$
  are  the restrictions to $X_1$ of the original measure and transformation.
  \end{lemma}

\begin{proof}   
 First we observe that
  \begin{align}\label{E:Disinvariance}
T^{-1}(Dis(d_T))\subset Dis(d_T).
 \end{align}
In fact, if $Tx\in Dis(d_T)$, then $\mathcal D^{d_T}(Tx)>0$.
Since $d_T$ is $1$-Lipschitz, by Lemma~\ref{L:Dineq}, $\mathcal D^{d_T}(x)>0$,
so $x\in Dis(d_T)$.  Therefore $T$ can be restricted to a transformation 
on the positively invariant set   $X_1=X\setminus Dis(d_T)$. 

Since $T$ is conservative and ergodic it follows from \eqref{E:Disinvariance} that
$\mu(Dis(d_T))=0$ or $\mu(Dis(d_T)^c)=0$. If it were the case that
$\mu(Dis(d_T)^c)=0$ then there would exist $r>0$ such that
$\mathcal D^{d_T}(x)>r$ on a set of positive measure, hence by Lemmas \ref{L:Dineq} and \ref{L:subinvatiantfunc}, as $T$ is conservative and
ergodic, the condition holds for  a.e. $x$, but this contradicts
the hypothesis by the Remark following Proposition~\ref{pro:mild-meas-sens2}. Therefore $\mu(Dis(d_T))=0$
and $X_1$ is a set of full measure. (It follows also that
$\mathcal D^{d_T}(x)=0\text{ for a.e. }x\in X.$)

Clearly, $d_T$ is a metric
on $X_1$. To see that it is $\mu$-compatible we calculate,
for $x\in X_1$ and $\ve>0$, 
\[
\mu(B^{d_T}(x,\ve)\cap X_1)=\mu(B^{d_T}(x,\ve))>0.
\]

  \end{proof}

\begin{remark}
\rm{In relation to Lemma~\ref{lem:when-d'-is-mu-comp} , we note that 
it is possible that a system $(X,\mu,T)$ is not W-measurably
sensitive, but does not itself admit
any $\mu$-compatible metric $d$ that is $1$-Lipschitz.
For example, consider the dynamical system $(I,\lambda,T)$ where $I$
is the unit interval and $\lambda$ is the Lebesgue measure. Let
$\alpha$ be a fixed irrational number between 0 and 1. For any
$x\in I$, we define:\[ T(x)=\begin{cases}
  x & \text{if }x=n\cdot\alpha+m\text{ for some }n,m\in\mathbb{Z}\\
  x+\alpha\pmod1 & \text{otherwise.}\end{cases}\]
This system is ergodic and not measurably sensitive as it is measurably
isomorphic to a rotation. However, there is no $\lambda$-compatible
$1$-Lipshitz metric on $I$.

Indeed, suppose that there is a $\lambda$-compatible metric $d$ such that
$d(Tx,Ty) \leq d(x,y)$ for all $x,y\in I$. Let $B$ be a ball
of radius $\alpha/2$ around 0. Since $d$ is $\lambda$-compatible, $B$
must have positive measure. Furthermore, since $T(0)=0$, for any point
$b \in B$, we must have $d(T(b),0) \leq d(b,0) < \alpha/2 $ and,
therefore, $T(b) \in B$. So  $T$ maps a set of positive measure into
itself. This is impossible for a transformation isomorphic mod 0
to an irrational rotation.
}
\end{remark}

\bigskip
In the rest of this section, we describe some useful consequences of a 1-Lipshitz metric being $\mu$-compatible.

Let $(X,d)$ be a metric space and $T:X\to X$ a transformation. Let 
$\omega(x)$ denote the set of accumulation  points of the positive orbit 
$\{T^nx: n\in\mathbb N_0\}$.  A point $x\in X$ is a {\it transitive point} for $T$ if
$\omega(x)=X$. When $(X,d)$ has no isolated points this is equivalent to the  
  (positive) orbit of $x$  being dense in $X$. As we will only consider 
  $\mu$-compatible metrics where $\mu$ is nonatomic, all our metric
  spaces will have no isolated points.
 $T$ is 
{\it transitive} if it has a transitive point. 
The transformation $T$ is {\it minimal} if $\omega(x)=X$ for all $x\in X$.
It is {\it uniformly rigid } if there exists a sequence $n_i$ such that 
$d(T^{n_i}x,x)$ converges to $0$ uniformly on $X$.

The following lemma is essentially known.

\begin{lemma}\label{L:transitive}
Let $(X,\mu,T)$ be a conservative and
  ergodic nonsingular dynamical system.  If $d$ is  a 
  $\mu$-compatible metric on $X$, then
  $\mu$-a.e. point of $X$ is transitive. \end{lemma}

\begin{proof} Since by assumption $\mu$ is nonatomic, 
$d$ has no isolated points.  By Proposition~\ref{P:metricprop}, $(X,d)$ is
separable, so there exist $\{x_i: i\in\mathbb N\}$ dense in $X$. For each $r\in \Q, r>0,$ and 
each $i,N\in\N$, set
\[ A^*_{i,N,r}=
\bigcup_{n\geq N}T^{-n}( B^d(x_i,r)).\]
Since $T$ is conservative 
and ergodic, each $ A^*_{i,N,r}$ is of full measure. Finally let \[
B=\bigcap_{i,N,r}
 A^*_{i,N,r}.\] Clearly $B$ is of full measure and each point in $B$ has a dense orbit.
  \end{proof}
  
  The following proposition is essentially from \cite{A-G01}. 

\begin{proposition}
\label{pro:uniformlyrigid} Let $(X,d)$ be a metric space 
and let $T:X\to X$ be a $1$-Lipschitz transformation. If $T$
is transitive, then it is a uniformly rigid, minimal isometry.
\end{proposition}

\begin{proof} Let $x$ be a point such that $\omega(x)=X$.
(This in particular implies that the metric $d$ is separable.)
Let $\ve>0$. There exists an integer $k>0$ such that
$d(x,T^kx)<\ve$.  Since $T$ is  $1$-Lipschitz,
for all $n\in\mathbb N$, $d(T^nx,T^n(T^kx))<\ve$. 
Let $y\in X$. Since $T^k$ is continuous,
for $n$ such that $d(y,T^nx)$ is sufficiently small,
$d(T^ky,T^k(T^nx))<\ve$. Then
\begin{align*}
d(y,T^ky)&\leq d(y,T^nx)+d(T^nx,T^n(T^kx))+d(T^k(T^nx),T^ky)\\
&<3\ve.
\end{align*}
Therefore $T$ is uniformly rigid. Now, in this case there exists
a sequence $n_i\to\infty$ such that $d(T^{n_i}x,x)\to 0$ for
all $x\in X$. Therefore, for all $x,y\in X$,
\[
0\leq d(T^{n_i}x,T^{n_i}y)-d(x,y) \leq  d(T^{n_i}x,x)+ d(y,T^{n_i}y) \to 0.
\]
If $T$ were not an isometry there would exist $x,y\in X$,
such that $d(Tx,Ty)<d(x,y)$, but then $d(T^{n_i}x,T^{n_i}y)$
could not converge to $d(x,y)$.

Finally we show that $T$ is minimal. Again, let $\omega(x)=X$
and $y\in X$. Let $\ve>0$, $z\in X$. There exists $i\in \mathbb N$
such that $d(T^ix,y)<\ve$. Then we can choose $j\in \mathbb N$
so that $d(T^{i+j}x,z)<\ve$. Then
\begin{align*}
d(T^jy,z)&\leq d(T^jy,T^{j+i}x)+d(T^{j+i}x,z)\\
&\leq d(y,T^{i}x)+d(T^{j+i}x,z)<2\ve.
\end{align*}
Therefore $\omega(y)=X$.
\end{proof}

Now, let $\mathcal{C}_d(X,X)$ be the space of continuous maps from $X$ to itself, with the metric $d(S_1, S_2) = \sup_{x \in X} \{d(S_1 x, S_2 x)\}$. We also define a subset \[
\mathcal{J}_T = \left\{S\in \mathcal C_d(X,X) \mid S \circ T = T \circ S\right\}.
\]
This is clearly a sub-semigroup of $\mathcal{C}_d(X,X)$ under composition.

The following proposition is essentially from \cite{A-G01}. We are indebted to Ethan Akin for the proof.

\begin{proposition}\label{P:invertibleisometry}
Let $(X,d)$ be a metric space and let $T$ be a  transitive and  $1$-Lipshitz transformation.  Then, for each $x\in X$, the evaluation map \[ev_x: \mathcal{J}_T \to X\text{ defined by }S \mapsto Sx\] is an isometry. Also, the space $\mathcal{J}_T$ is the closure of sequence $\{id, T, T^2, \ldots\}$ in $\mathcal C_d(X,X)$. 
If in addition  the metric space $(X,d)$ is complete, then the evaluation map $ev_x$ is an
invertible isometry. Moreover, the semigroup $\mathcal{J}_T$ is then a group, and therefore $T \in \mathcal{J}_T$ has to be invertible. 
\end{proposition}

\begin{proof}
Fix a point $x \in X$ and let $S, S'\in \mathcal{J}_T$. We wish to show that the map $ev_x$ is an isometry. Since $S$ and $S'$ both commute with $T$, and $T$ is $1$-Lipshitz,  for all $m$,\[
d\left( S(T^m x) ,  S'(T^m x) \right)  \leq d(Sx,S'x).
\]
Since $S$ and $S'$ are both continuous and the set of all $\{T^m x\}$ is dense,  for all $y \in X$, $d(Sy,S'y) \leq d(Sx,S'x)$ and therefore \[
d(S,S')_{\mathcal C_d(X,X)} = \sup_{y\in X}d(Sy,S'y)_X = d(Sx, S'x)_X = d(ev_x S, ev_x S')_X
\] and so $ev_x$ is an isometry.

Now, the subset $\mathcal{J}_T$ is clearly closed in $\mathcal C_d(X,X)$. Fix some $S \in \mathcal{J}_T$ and $x \in X$. Since $T$ is minimal, $x$ is a transitive point, and so there is a sequence $\{n_j\}$ such that $\lim_{j \to \infty} T^{n_j}x = Sx$. In other words, $\lim_{j \to \infty} ev_x T^{n_j} = ev_x S$ in $X$. Since $ev_x$ is an isometry, this implies that $\lim_{j \to \infty} T^{n_j} = S$ in $\mathcal C_d(X,X)$, completing the proof of the first part of the proposition.

If we assume that the space $(X,d)$ is complete, so is the space $\mathcal C_d(X,X)$. For   $x\in X$, we show that $ev_x$ is surjective. 

Pick a $y\in X$. There is a sequence of $n_j$-s such that $T^{n_j}x\to y$. In particular, the sequence $ev_x\left( T^{n_j} \right)$ is Cauchy. Since $ev_x$ is an isometry, the sequence  $T^{n_j}$ is Cauchy in $\mathcal C_d(X,X)$. By completeness, it has a limit $S\in \mathcal J_d$ (since $\mathcal J_d$ is closed); clearly $ev_x S = y$ and $ev_x$ is surjective.

Now, let $S\in \mathcal J_d$ be arbitrary. Since the map $ev_{Sx}$ is surjective, we can pick an $S'$ so that \[
S'(Sx) = ev_{Sx}S' = x.
\] Since $ev_x (SS') = (S'S)x = x$ and $ev_x$ is injective, $S\circ S'$ is the identity, and $S' = S^{-1}$. So, all maps in $\mathcal J_d$ are invertible.
\end{proof}

%%%%%
\section{\label{sec:Constructing-mu-comp.-metr-on-isom.-copies}
W-measurable sensitivity on isomorphic mod 0 dynamical systems}

We prove that W-measurable sensitivity is invariant
under measurable isomorphism. Here we use that we are working on standard Borel spaces.

\begin{lemma} \label{lem:extension-of-mu-comp-metrics} Let $(X,\mathcal S)$
  be a standard Borel space, with $\mu$ a nonatomic measure on $\mathcal S$. Let
  $U\subset X$ be a Borel subset of full measure and
  let $d$ be a $\mu$-compatible metric defined on $U$. 
  Then the metric $d$ can be extended to a $\mu$-compatible metric $d_1$ on all
  of $X$ in such a way that $d$ and $d_1$ agree on a set of full
  measure.
\end{lemma}

\begin{proof}
Since the measure is nonatomic and $U$ is Borel, it must have the same cardinality as $X$. 
Using e.g.  \cite[3.4.23]{Sri98}  one can show that there exists a Borel set $Z\subset U$ 
of measure zero and cardinality $c$. Therefore there exists a Borel isomorphism
$\phi: (X\setminus U)\sqcup Z\to Z$. Then we can define $\phi':X\to U$ by
 \begin{align*}
  \phi'(x)=\begin{cases}
    \phi(x) & \text{if } x\in(X \setminus U)\cup Z;\\
    x & \text{if }  x\in U\setminus Z.\end{cases}
    \end{align*}
($\phi'$  is the identity on the full-measure Borel subset $U\setminus Z$.) 
  For $x,y\in X$  define $d_1(x,y) = d(\phi'(x),\phi'(y))$.
  Clearly, since $d$ is a measurable metric, so is $d_1$. Since every $d_1$-ball
  corresponds to a $d$-ball under the map $\phi$, which is a
  Borel isomorphism, $d_1$ is also a $\mu$-compatible metric 
  and agrees with $d$ on $(U\setminus Z)\times (U\setminus Z)$.
\end{proof}

Using Lemma~\ref{lem:extension-of-mu-comp-metrics}, we can prove the
invariance of W-measurable sensitivity.

\begin{proposition}\label{prop:light-meas-sens-and-meas-isomorph}
  Suppose $(X,\mu ,T)$ is a W-measurably sensitive nonsingular  dynamical system.
  Let $(X',\mu' ,T')$ be a nonsingular dynamical system isomorphic mod 0 to
  $(X,\mu ,T)$.  Then, $(X',\mu' ,T')$ is also W-measurably sensitive.
\end{proposition}

\begin{proof}
  Suppose $(X',\mu' ,T')$ is not W-measurably sensitive. Then, there
  is a $\mu'$-compatible metric $d'$ on $X'$ such that $(X',\mu' ,T')$
  is not W-measurably sensitive with respect to $d'$.

  By the definition of measurable isomorphism, there must be Borel subsets
  $U\subset X$ and $U'\subset X'$ and a measure-preserving bijection
  $\phi:U\to U'$ such that
  $\mu\left(X\setminus U\right)=\mu'\left(X'\setminus U'\right)=0$,
  and $\phi\circ T=T'\circ\phi$.

  We define a metric $d$ on $U$ by $d(x,y) = d'(\phi(x),\phi(y))$ for
  $x,y \in U$. It is clearly $\mu$-compatible on $U$. We  apply Lemma~\ref{lem:extension-of-mu-comp-metrics} to
  extend $d$ to a $\mu$-compatible metric $d_1$ defined on all of $X$
  that agrees with $d$ almost everywhere.

  Now, we  show that $(X,\mu,T)$ is not W-measurably sensitive
  with respect to $d_1$. Let $\delta > 0$. Since $(X',\mu',T')$ is not
  W-measurably sensitive with respect to $d'$, by part $(3)$ of
  Proposition~\ref{pro:mild-meas-sens2}, there must be an $x'\in X'$
  such that the set
  $Y'= \{y \in X': \forall n \geq 0, \: d'(T'^{n}x,T'^{n}y) <
  \delta/2 \} $
  has positive measure. Let $Y$ be the corresponding set in $X$, that
  is $Y = \phi^{-1}(Y'\cap U')$. Note that $\mu(Y) = \mu'(Y') >0$.

  Pick any $x \in Y$. By the triangle inequality, for all $y\in Y$ and
  all integers $n$, we have:
  \begin{align*}
    d_1(T^n x,T^n y) &= d'(T'^n(\phi(x)),T'^n(\phi(y)) \\
    &\leq d'(x',T'^n(\phi(x)) + d'(x',T'^n(\phi(y))\\ &\leq \delta.
  \end{align*}
  Since $Y$ has positive measure, $(X,\mu,T)$ cannot be W-measurably
  sensitive. 
\end{proof}

\begin{proposition}
  \label{thm:non-meas-sens2} Let $(X,\mu,T)$ be a conservative
  and ergodic nonsingular  dynamical system. $T$ is W-measurably
  sensitive if and only if all measurably isomorphic dynamical
  systems
  $(X',\mu',T')$ admit no $\mu'$-compatible metrics that are $1$-Lipshitz.
\end{proposition}

\begin{proof}  First we note that if a dynamical system $(X',\mu',T')$
admits a $\mu'$-compatible $1$-Lipshitz metric $d'$, then this system
could not be W-measurably sensitive, since for  all integers $n$,
$d'(T^n x,T^n y) \leq d'(x,y)$.
Now, if 
a dynamical system $(X,\mu,T)$ is W-measurably sensitive,  then every 
measurably isomorphic system $(X',\mu',T')$ will also be W-measurably
sensitive, and therefore will not admit a $\mu'$-compatible
$1$-Lipshitz metric $d'$.

For the converse, suppose $(X,\mu,T)$ is \emph{not} W-measurably sensitive. 
By 
Lemma~\ref{lem:when-d'-is-mu-comp}  there is  a set of full measure  $X_{1} \subset X $, such that if $T_{1}$ is 
 $T$ restricted to $X_{1}$ and $\mu_{1}$
to be $\mu$ restricted to $X_{1}$,  then $d_{T}$ is a $1$-Lipshitz
$\mu_{1}$-compatible metric on $(X_{1},\mu_{1},T_{1})$.
\end{proof}

\noindent{\bf Remark.}   If $d$ is
  a $\mu$-compatible metric on $(X,\mu)$, $X$ must be a separable
  metric space under $d$ \cite{J-S08} and Proposition[(3)]~\ref{P:metricprop}, so  $X$ has
  at most the cardinality of the reals. A nonatomic (probability) Lebesgue space
  is defined as a measure space $(X,\mathcal S,\mu)$ that is 
  isomorphic mod 0 to the unit interval $I$ with Lebesgue measure $\lambda$, 
  i.e., there exists sets of full measure $U\subset X$ and $U'\subset I$
  such that there is a (measure-preserving) isomorphism from $U$ to $U'$.
  However, there is not restriction on $X\setminus U$ other than it is of $\mu$-measure 0
  and it could have cardinality greater than the reals. In this case
    $X$ would admit no $\mu$-compatible metric, and for
  instance,    transformations on this space would be  vacuously
  W-measurably sensitive. We introduce the following definition for
  Lebesgue spaces.

\begin{definition}
  \label{def:VW-meas-sens} Let $(X,\mu)$ be a Lebesgue space 
  (or more generally a $\sigma$-finite measure space) and let
  $T$ be a nonsingular transformation on $(X,\mu)$. A
   dynamical system $(X,\mathcal S, \mu,T)$ is \emph{VW-measurably
    sensitive}  if
    for every positively invariant measurable set of full measure set $U\subset X$,
    the system $(U,\mathcal S(U),\mu,T)$ is W-measurably
    sensitive.
    \end{definition}

\noindent{\bf Remark.}   (1)  By Lemma~\ref{lem:extension-of-mu-comp-metrics}, on standard Borel spaces, the notions of W-measurable 
sensitivity and VW-measurable sensitivity are equivalent. Also,  it follows from the definition that VW-measurable sensitivity is invariant under isomorphism.

(2) Here we note that nonsingular  dynamical system (on standard Borel spaces)   $(X,\mu ,T)$ 
do admit $\mu$-compatible measures. If fact we know that if $(X,\mathcal S)$ 
is a standard Borel space and $\mu$ is a continuous measure on $\mathcal S$, which we may assume a probability measure, then there exists
a Borel isomorphism $\phi$  from  $(X,\mathcal S,\mu)$ to the unit interval with Lebesgue 
measure $(I,\mathcal B,\lambda)$ (see e.g. \cite[3.4.23]{Sri98}.  
Clearly Euclidean distance $d$ on $I$ is a $\lambda$-compatible measure on 
 $(I,\mathcal L,\lambda)$. Then $d'$ defined by $d'(x,y)=d(\phi(x),\phi(y))$ is 
 a $\mu$-compatible metric on $X$.

%%%%%%
\section{Characterization of W-measurable Sensitivity}\label{sec:meas-sens-fin-measure}

 We shall prove our main result, that such a
transformation is either W-measurably sensitive or measurably
isomorphic to a minimal uniformly rigid isometry. 
This can be seen as a measurable version of 
the Dichotomy Theorem of Auslander and Yorke \cite{A-Y80} for topological dynamical systems
(continuous surjective maps on compact metric spaces), which states that
a transitive  map on a topological system is either sensitive or almost equicontinuous.
Related topological dynamical results are in \cite{GW93}, \cite{A-G01} and the references therein.

\begin{theorem}\label{T:nonsingular}
Let $(X,\mu,T)$ be a conservative and ergodic nonsingular dynamical system. 
Then $T$ is either W-measurably
  sensitive or $T$ is isomorphic mod 0 to
 an invertible minimal uniformly rigid isometry on a Polish space.
\end{theorem}

\begin{proof} 
Suppose $T$ is not W-measurably sensitive.  Then, 
by Lemma~\ref{lem:when-d'-is-mu-comp}, there exists a positively invariant set $X_1$ of full measure
such that $d_T$ is $\mu$-compatible for the system $(X_1,\mu_1,T_1)$, where
$\mu_1$ is the restriction of $\mu$ to $X_1$ and $T_1$ the restriction of $T$ to $X_1$. 
By Lemma~\ref{L:transitive}, $T_1$ is transitive with respect to $d_T$.
Since $T_1$ is 1-Lipshitz with respect to $d_T$, by Proposition~\ref{pro:uniformlyrigid},  $T_1$ is a 
uniformly rigid minimal isometry on $(X_1,d_T)$.

Now, let $(X_{2},d_{2})$ be the topological completion of the metric space
$(X_1,d_T)$. Since $d_T$ is separable, $d_2$ is also separable so $(X_{2},d_{2})$ is Polish.
We extend the measure $\mu_1$ to $X_{2}$ by defining a set $S\subset X_{2}$
to be measurable if $S\cap X_1$ is measurable, with $\mu_{2}(S)=\mu_1(S\cap X)$.
Since $T_1$ is an isometry, it is continuous on $(X_1,d_T)$, so there is a unique way to extend it to a continuous transformation
$T_{2}$ on $(X_{2},d_{2})$. It's easy to verify that $T_{2}$ must also
be an isometry with respect to $d_2$. It is invertible by Proposition~\ref{P:invertibleisometry}.

Clearly, the dynamical system $(X_{2},\mu_{2},T_{2})$ is measurably
isomorphic to $(X,\mu,T)$. 
\end{proof}

 Invertible examples of W-measurably sensitive transformations are  mentioned in Section~\ref{sec:invertible-examples}, but we have the following direct consequence of the theorem.

\begin{corollary}  If a conservative and ergodic nonsingular transformation
is not invertible a.e. then it cannot be isomorphic mod 0 to an invertible isometry, so it must 
be W-measurably sensitive.
\end{corollary}

As a first application of Theorem~\ref{T:nonsingular}, we show
the following proposition.

\begin{proposition} \label{pro:equiv-weak-and-light} If a dynamical
  system is W-measurably sensitive, then it is measurably
  sensitive. If a dynamical system is conservative
  ergodic and measurably sensitive, then it is W-measurably
  sensitive.
\end{proposition} 
\begin{proof}
  First, suppose $(X,\mu,T)$ is a W-measurably sensitive
  nonsingular dynamical system.  
By
  Proposition~\ref{prop:light-meas-sens-and-meas-isomorph}, every 
  isomorphic mod 0 dynamical system $(X_1,\mu_1,T_1)$ is also
  W-measurably sensitive. %
   So, for any $\mu_1$-compatible metric $d_1$ on $X_1$, there is a
  $\delta >0$ such that for all $x\in X_1$, we have
  $\limsup_{n\to \infty}d(T^n x, T^n y)>\delta$ for almost all
  $y\in X_1$.

  In particular, 
  \[\mu_1 \{y\in B^{d_1}(x,\varepsilon):
  \exists n>0 \text{ with } d_1(T_1^n(x),T_1^n(y)) > \delta \} 
  = \mu_1 \left( B^{d_1}(x,\varepsilon) \right) > 0.\]
  This implies that  there is  an $n>0$ for which
  the set \[\{y\in B^{d_1}(x,\varepsilon): d_1(T_1^n(x),T_1^n(y)) > \delta \}\]
   has positive measure. Thus $(X,\mu,T)$ is
  measurably sensitive.
  
To show the convere, suppose $(X,\mu,T)$ is a
  conservative and ergodic dynamical system that is not W-measurably
  sensitive. Then, by Theorem~\ref{T:nonsingular}, there is a
  isomorphic mod 0 dynamical system $(X_1,\mu_1,T_1)$ and a
  $\mu_1$-compatible metric $d_1$ on $X_1$ that is an isometry. For
  all $\delta>0$, choose any $\varepsilon<\delta$, and then for
  any $x\in X_1$ with $d_1(x,y)<\varepsilon$, for all integers $n$, $ d_1(T_1^n(x),T_1^n(y)) =d_1(x,y) <\delta$.
  So  neither $(X_1,\mu_1,T_1)$ nor $(X,\mu,T)$ can be measurably
  sensitive.
\end{proof}

\begin{remark}
\rm{  Note that the assumption that the dynamical system is ergodic is
  crucial to the above statement. For example, as we mentioned in
  section~\ref{sec:remark-nonatomic}, no transformation can be
  W-measurably sensitive on a space with points of positive
  measure. Nonetheless, there are (non-ergodic) transformations on
  such spaces which are measurably sensitive according to the
  definition in \cite{J-S08}.
 }

\end{remark}

In the case when the measure space is finite (and a conservative transformation is measure-preserving), we can prove more.

\begin{theorem}\label{thm:non-meas-sens}
  Let $(X,\mu,T)$ be a finite measure-preserving ergodic dynamical system. 
  Then $T$ is either W-measurably
  sensitive  or $T$ is isomorphic to
  a minimal, uniformly rigid  compact group rotation (i.e., a Kronecker transformation).
\end{theorem}

\begin{proof}

   We first  show that $X$ is a totally bounded space with respect to
  any $\mu$-compatible metric $d$ that is an isometry for $T$.

 Let $\varepsilon>0$.  Let $C=\mu(B(x,\frac{\varepsilon}{2}))$
  for some $x_0\in X$. Since the metric is $\mu$-compatible,
  $C>0$.  We claim that 
  this is a constant independent of $x$. In fact, if we let $f(x)=\mu B(x,\frac{\varepsilon}{2})$.
  Then, 
  \[f(Tx)=\mu(B(Tx,\frac{\varepsilon}{2}))=\mu(T^{-1} B(x,\frac{\varepsilon}{2}))=\mu(B(x,\frac{\varepsilon}{2}))
  =f(x).\]
As $f$ is continuous and $T$ is transitive (Lemma~\ref{L:transitive}), $f$ is constant.

  Now choose  a largest possible collection of points $\{x_1,\ldots,x_n\}$
  such that the balls $B(x_i,\frac{\varepsilon}{2})$ are all
  disjoint. Note that the size of  any such collection will be no greater
  than $\frac{\mu(X)}{C}$, as they all have the same measure.
  %(as the quantity   $\mu\left(\bigcup_{i=1}^n B_{\frac{\varepsilon}{2}}(x_i)\right) = n\cdot C$
  %cannot be greater than $\mu(X)$). 
  By the triangle inequality, for
  any point $x \in X$, there must be an $i$ such that
  $d(x,x_i)<\varepsilon$, as otherwise the ball
  $B(x,\frac{\varepsilon}{2})$ would be disjoint from all the
  $B(x_i,\frac{\varepsilon}{2})$'s. So 
  $X=\bigcup_{i=1}^n B(x_i,\varepsilon)$.  Since $\varepsilon$
  was arbitrary, $X$ is totally bounded.

Now, as we have seen before, if $T$ is not W-measurably sensitive,
there exists a positively invariant set $X_1$ of full measure
such that $d_T$ is $\mu$-compatible for the system $(X_1,\mu,T)$,
and $T$ is a minimal uniformly rigid isometry.
Let $(X_{2},d_{2})$ be the topological completion of the metric space
$(X_1,d)$. It is complete and totally bounded, and therefore compact. 
As before, we extend the measure $\mu$ to $X_2$, and $T$ extends to 
 a continuous transformation
$T_{2}$ on $(X_{2},d_{2})$ that is an
 isometry with respect to $d_2$.

Clearly, the dynamical system $(X_{2},\mu_{2},T_{2})$ is measurably
isomorphic to $(X,\mu,T)$, and $T_{2}$ is an ergodic isometry on the
compact metric space $X_{2}$, as desired. Finally, every $d_{2}$-ball
is measurable and contains a $d$-ball, so the metric $d_{2}$ is
$\mu_{2}$-compatible.

\end{proof}

\noindent{\bf Remark.} 
Theorems 1 and 2 also hold for VW-measurable sensitivity.

\appendix

\section{Connections to pairwise sensitivity and other literature}\label{sec:pairwise-sensitivity}

In their paper \cite{cadre2005ps}, Cadre and Jacob introduce the
notion of \emph{pairwise sensitivity}, which they define as follows.
They only consider finite measure-preserving transformations, 
so we will restrict to them in this appendix. 

\begin{definition}
  Let $(X,\mu)$ be a Lebesgue probability space and let us fix
  a metric $d$ on $X$.
	
  An endomorphism $T$ is said to be \emph{pairwise sensitive} (with
  respect to initial conditions) if there exists $\delta > 0$ --- a
  sensitivity constant --- such that for $\mu^{\otimes2}$-a.e.
  $(x,y)\in X \times X$, one can find $n\geq 0$ with
  $d(T^n x,T^n y) \geq \delta$.
\end{definition}

Since this concept depends on the choice of a metric $d$, we will
often refer to $T$ as being \emph{pairwise sensitive with respect to
  $d$}.

Cadre and Jacob prove that weakly mixing finite measure-preserving  transformations are pairwise
sensitive, and that a certain entropy condition implies pairwise
sensitivity for ergodic transformation.
  
This notion is very closely related to the notion of W-measurable
sensitivity, as the following proposition shows.
  
\begin{proposition}
  Let $(X,\mu,T)$ be a dynamical system and $d$ be a $\mu$-compatible metric on $X$.
  Then $T$ is pairwise sensitive with respect to $d$ if and only if
  it is W-measurably sensitive with respect to $d$.
\end{proposition}

\begin{proof} 
  First  suppose that the system is W-measurable sensitive with
  respect to $d$. Then,
  there is a $\delta>0$ such that for every $x \in X$ the set
  \[
  Y_x = \{y \in X:  \exists n \text{ such that } d(T^n x,T^n y) >
  \delta \}\]
  has full measure. By Fubini's theorem (for the version we use,
  see~\cite{evans1992mta}), the set
  $Y = \{ (x,y) \in X\times X:  \exists n \text{ such that } d(T^n x,T^n
  y) > \delta \}$
  must have full $\mu^{\otimes 2}$-measure in $X \times X$. So  $T$ is
  pairwise sensitive with respect to $d$.

  Now, suppose that the system is pairwise sensitive with respect to a
  $\mu$-compatible metric $d$. That is, there is a $\delta > 0$ such
  that the set $Y$, defined as before, has full $\mu^{\otimes
    2}$-measure in $X\times X$.
  
  Take any $x \in X$. We claim that for almost every $y \in X$, there
  is an $n$ such that $d(T^n x, T^n y) >\delta/2$. Once we have this
  claim, Proposition~\ref{pro:mild-meas-sens2} implies that $T$ is
  W-measurably sensitive with respect to $d$.
  
  To prove the above claim, we need to show that the set
  $S_x = \{y \in X:  \forall n,\, d(T^n x,T^n y) \leq \delta/2\}$
  has measure zero. Take any $y_1,y_2 \in S_x$. By the triangle
  inequality, for all $n$ we have $d(T^n y_1,T^n y_2) \leq \delta$.
  So  the pair $(y_1,y_2)$ does not belong to the set
  $Y \subset X \times X$. In other words, the Cartesian product
  $S_x \times S_x$ lies wholly inside the $\mu^{\otimes
    2}$-measure-zero
  set $(X \times X)  \setminus Y$. Again by Fubini, this is only possible if the set
  $S_x$ is measurable and has $\mu$-measure zero.
\end{proof}

  With this in mind, our Theorem~\ref{thm:non-meas-sens} implies the
  following theorem concerning pairwise sensitivity.

\begin{theorem} \label{thm:classify-pairwise-sensitivity} 
  Let $(X,\mu,T)$ be a nonatomic ergodic finite measure-preserving
  dynamical system. Suppose further
  that this dynamical system is not isomorphic mod 0 to a
  Kronecker transformation. Then, for any $\mu$-compatible metric $d$,
  $T$ is pairwise sensitive with respect to $d$.
\end{theorem}

%Note that many of the additional assumptions we need to make do not
%significantly narrow down the applicability of this result. Cadre and
%Jacob also work in the finite measure-preserving setting, and their main result also applies only to
%ergodic transformations. Just like we showed with W-measurable
%sensitivity in the Remark in Section~\ref{sec:remark-nonatomic}, if a
%measure space has points of positive measure, no transformation on it
%can exhibit pairwise sensitivity with respect to any metric.

We do need to assume that the metric $d$ is $\mu$-compatible. However,
while Cadre and Jacob never specify any restrictions on their metric,
they also tacitly use several very similar properties. For example,
they extensively use the notion of the \emph{support} of a measure
(i.e., the complement of the largest open set of zero measure), which
is not well-defined without the assumptions that open and closed sets
are measurable, and that the space is separable (if the space
  were not separable, the union of all open sets of measure zero 
  may have positive measure even if  measurable). Together, these
two properties are almost sufficient to force the metric to be
$\mu$-compatible, as the following proposition shows.

\begin{proposition}\label{P:three_eq}
  A metric $d$ on a measure space $(X,\mu)$ is $\mu$-compatible if and
  only if the following three conditions are satisfied:
  \begin{enumerate}
  \item Every $d$-ball is $\mu$-measurable.
  \item The space $X$ is separable under $d$.
  \item The support of $\mu$ is the whole of $X$.
  \end{enumerate}
\end{proposition}

\begin{proof}
  The fact that if $d$ is $\mu$-compatible then $X$ is separable under
  $d$ is shown by~\cite{J-S08}. The other two properties are obvious.

  Now, suppose that $d$ satisfies all of the first two properties.
  Then, the notion of the support is well-defined. Clearly, the
  support of the measure is the whole space if and only if every
  non-empty open set has positive measure, i.e., if $d$ is
  $\mu$-compatible.
\end{proof}
According to this Proposition~\ref{P:three_eq}, to go from $\mu$-compatible metrics to
the metrics Cadre and Jacob use, we only need to require that the
support of the measure is the whole space. This can always be
achieved by removing a set of measure zero from the space.

 With this assumption,
Theorem~\ref{thm:classify-pairwise-sensitivity} sharpens the results
of~\cite{cadre2005ps}.

\bigskip

We also mention a recent work that we learned of  from Ethan Akin after the research for this paper was completed.
  In \cite{H-L-Y2010}, Huang, Lu, and Ye introduce the notion of $\mu$-sensitivity for topological 
  dynamical systems and study  its properties, and in particular show that it is equivalent to 
  pairwise sensitivity \cite[2.4]{H-L-Y2010}.  We note that in \cite{H-L-Y2010},  the authors consider 
  topological dynamical systems (continuous maps on compact metric spaces) and put invariant
  probability measures on them, while we consider measurable dynamical systems (nonsingular
  maps on standard Borel spaces) and put compatible metrics on them. We note also, one of the theorems
  of Huang, Lu, and Ye, \cite[Theorem 5.4]{H-L-Y2010} is related to our Theorem 2.
%%%%%%%%%%%%%%%%%%%%%%%%%%%%%%%%%%%%%%%%%%%%%

%Bibliography
\bibliographystyle{amsalpha}
\bibliography{MetricSen_Bib}

%%%%%%%%%%%%%%%%%%%%%%%%%%%%%%%%%%%%%%%%%%%%%

\end{document}